\documentclass[reqno,twoside,11pt,english]{amsart}
\usepackage{amsmath,amsfonts,amssymb,amsthm,epsfig}
\newcommand{\RNum}[1]{\uppercase\expandafter{\romannumeral #1\relax}}
\usepackage{comment}
\usepackage{color}
\usepackage[final,colorlinks,linkcolor=blue,anchorcolor=red,citecolor=blue]{hyperref}
\usepackage{graphicx}
\usepackage{titletoc,epsf}
\usepackage{amsmath,amsfonts,latexsym,amsthm,amsxtra,amssymb,bbm}
\allowdisplaybreaks
\usepackage{graphicx,float}

\allowdisplaybreaks
\let\f=\frac

\let\th=T
\let\pa=\partial

\newtheorem{theorem}{Theorem}[section]
\newtheorem{lemma}[theorem]{Lemma}
\newtheorem{proposition}[theorem]{Proposition}

\newtheorem{remark}[theorem]{Remark}

\numberwithin{equation}{section}

\numberwithin{equation}{section}

\topmargin - .23 in

\begin{document}

\title[Limit of energy-critical G-L equation]{The limit theory of  energy-critical complex Ginzburg-Landau equation}

\author{Xing Cheng, Chang-Yu Guo and Yunrui Zheng}

\address[X. Cheng]{Key Laboratory of Hydrologic-Cycle and Hydrodynamic-System of Ministry of Water Resources and school of mathematics, Hohai University, Nanjing 210098, Jiangsu, P. R. China}
\email{{\tt chengx@hhu.edu.cn}}

\address[C.-Y. Guo]{Research Center for Mathematics and Interdisciplinary Sciences, Shandong University, 266237,  Qingdao and  Frontiers Science Center for Nonlinear Expectations, Ministry of Education, P. R. China}
\email{{\tt changyu.guo@sdu.edu.cn}}

\address[Y. Zheng]{ School of Mathematics, Shandong University, Shandong 250100, Jinan, P. R. China.}
\email{{\tt yunrui\_zheng@sdu.edu.cn}}

\thanks{X. Cheng has been partially supported by the NSF of Jiangsu Province (Grant No.~BK20221497). C.-Y. Guo is supported by the Young Scientist Program of the Ministry of Science and Technology of China (No.~2021YFA1002200), the National Natural Science Foundation of China (No.~12101362), the Taishan Scholar Project and the Natural Science Foundation of Shandong Province (No.~ZR2022YQ01). Y. Zheng was supported by NSFC (Grant No.~11901350), and The Fundamental Research Funds of Shandong University.}

\begin{abstract}
In this article, we study the limit behavior of solutions to an energy-critical complex Ginzburg-Landau equation.
Via energy method, we establish a rigorous theory of the zero-dispersion limit from energy-critical complex Ginzburg-Landau equation to energy-critical nonlinear heat equation in dimensions three and four for both the defocusing and focusing cases. Furthermore, we derive the invisicid limit of energy-critical complex Ginzburg-Landau equation to energy-critical nonlinear Schr\"odinger equation in dimension four for the focusing case.

\bigskip


\bigskip

\end{abstract}

\maketitle

\section{Introduction}\label{se1}
In this article, we will consider the Cauchy problem for the energy-critical complex Ginzburg-Landau equation (CGL) in $\mathbb{R}^d$, $d=3, 4$,
\begin{align}\label{eq:g-l}
\begin{cases}
e^{-i\theta}v^\theta_t -  \Delta v^\theta +
\mu f(v^\theta)
= 0, \\
v^\theta(0,x) =  v^\theta_0(x),
\end{cases}
\end{align}
where $v^\theta$ is a complex-valued function of $(t,x) \in \mathbb{R}_+\times \mathbb{R}^d$,  $- \frac{\pi}2 < \theta < \frac{\pi}2$, $\mu = \pm 1$, and
$f \left(v^\theta \right) =
\left|v^\theta \right|^{\f4{d-2}}  v^\theta$.
When $\mu =1$, the equation \eqref{eq:g-l} is called \emph{defocusing}; while when $\mu=-1$, it is \emph{focusing}. Equation \eqref{eq:g-l} is the $L^2$ gradient flow of the energy functional
\begin{align*}
  E\left(v^\theta \right) := \int \frac12
  \left|\nabla v^\theta \right|^2 + \mu \frac{d-2}{2d} \left|v^\theta \right|^{\f{2d}{d-2}}
\,\mathrm{d}x.
\end{align*}
The CGL equation can be derived from the Navier-Stokes equation via multiple scaling methods in convection (see for instance \cite{NW}), and solutions to CGL equation can be derived as an amplitude of some good unknowns in the asymptotic expansion near some equilibria such as in Taylor-Coutte flow and plane Poiseuille flow (see \cite{AK,GJL,GL}). Compactness property of the Ginzburg-Landau type equation in dimension four was studied in \cite{WangCY-2004}.    

The global well-posedness (GWP) of \eqref{eq:g-l} in the defocusing case was proved by C. Huang and B. Wang \cite{HW}, where they used the energy induction method developed by J. Bourgain \cite{Bou} for the defocusing energy-critical nonlinear Schr\"odinger equation. In the focusing case $\mu = -1$, \eqref{eq:g-l} admits a stationary solution $W$, where
	\begin{align}\label{eq1.3v19}
		W(t, x) = W(x) :=  \frac1{ \left( 1 + \frac{  |x|^2}{ d(d-2) }  \right)^\frac{d-2}2 }.
	\end{align}
Motived by global well-posedness theory for energ-critical focusing  Schr\"odinger equations (NLS), very recently, the authors \cite{CGZ} studied \eqref{eq:g-l} in the focusing case, and established the global well-posedness  theory when the energy of the initial data is less than the energy of $W$ and the kinetic energy of the initial data is less than that of $W$, via  the concentration-compactness/rigidity method. 

In \eqref{eq:g-l}, the parameter $\theta \in \left( - \frac\pi2, \frac\pi2 \right)$ is a measure of the relative strength of the dispersive and dissipative effects in the equation: when $\theta \to 0, $ \eqref{eq:g-l} tends to the energy-critical nonlinear heat equation (NLH)
\begin{align}\label{eq:h}
\begin{cases}
u_t -  \Delta u + \mu f(u) = 0, \\
u(0,x) =  u_0(x).
\end{cases}
\end{align}
 When $|\theta| \to \frac\pi2$,
 \eqref{eq:g-l} tends to the energy-critical nonlinear Schr\"odinger equation (NLS)
\begin{align}\label{eq:s}
\begin{cases}
iv_t +  \Delta v + \mu f(v) = 0, \\
v(0,x) =  v_0(x).
\end{cases}
\end{align}

Before presenting the limit theory of solutions to energy-critical CGL, we would like to discuss some known results about the global well-posedness of the limit equations to energy-critical  CGL, that is energy-critical NLS and energy-critical  NLH.
\begin{enumerate}
\item[(i)]
The energy-critical nonlinear heat equation is the $L^2$ gradient flow of the energy
\begin{align*}
  E(u) := \int \frac12 |\nabla u|^2 + \mu \frac{d-2}{2d} |u|^{\f{2d}{d-2}}
\,\mathrm{d}x.
\end{align*}
The basic theory for the nonlinear heat equation can be found in \cite{QS,La}. Recently, the global well-posedness theroy for   energy-critical focusing nonlinear heat equations has been studied by S. Gustafson and D. Roxanas \cite{GR}.

\item[(ii)]
The global well-posedness theory for energy-critical nonlinear  Schr\"odinger equation has been studied intensively in the last two decades. The global well-posedness theory of energy-critical NLS in three and four dimensions is almost complete (see for instance \cite{Bou,CKSTT,D2019,KM,KV1,RV,V1}), except the GWP of non-radial focusing case in dimension three.
\end{enumerate}

Our main motivation of this article is to establish a rigorous theory for the limit $v^\theta \to u$ in an appropriate sense as $\theta \to 0$ in dimensions three and four. Meanwhile, we also consider the behavior of $v^\theta \to v$ as $|\theta |  \to \pi/2$ for the focusing case in dimension four. In our first main result, we consider the case when $\theta \to 0$, which corresponds to the zero-dispersion limit of \eqref{eq:g-l}. Based on the global well-posedness theory of CGL and NLH, we derive the following result on zero-dispersion limit.
\begin{theorem}[Zero-dispersion limit of energy-critical CGL]\label{thm:df_heat}
  Let $d= 3, 4$. Suppose that the initial data $v^\theta_0 \in H^1$ for \eqref{eq:g-l} and $ u_0 \in H^2$ for \eqref{eq:h} satisfy $\left\|v^\theta_0 - u_0 \right\|_{H^1} \le \varepsilon$ for a  sufficiently small universal $\varepsilon > 0 $. In the focusing case, we additionally assume that
  \begin{align*}
    \max\left\{E \left(v_0^\theta \right), E(u_0 )\right\} < E(W)
    \text{ and }
   \max\left\{\left\|\nabla v_0^\theta \right\|_{L^2}, \|\nabla u_0 \|_{L^2} \right\}< \|\nabla W\|_{L^2}.
   \end{align*}
   Then for any $T\in \left(0,\epsilon^2|e^{i\theta}-1|^{-2}\right)$, the solutions $v^\theta$ of \eqref{eq:g-l} and $u$ of \eqref{eq:h} satisfy
  \begin{align*}
    \left\|v^\theta - u \right\|_{L^\infty_t H^1_x([0, T] \times \mathbb{R}^d)}
    \le C_0   \left( \left\|v^\theta_0 - u_0 \right\|_{H^1} +  \left|e^{i\theta} - 1 \right|
    \left(T^{1/2} + 1 \right) \right),
  \end{align*}
  where $C_0 $ depends only on $v^\theta_0$ and $u_0$.
\end{theorem}

As far as we know, this seems to be the first result in the literature concerning the zero-dispersion of (variant of) the energy-critical nonlinear heat equation. As a direct consequence of Theorem \ref{thm:df_heat}, we infer that on any interval $[0, T] \subseteq \mathbb{R}_+$, if the initial datum satisfies the conditions in
Theorem \ref{thm:df_heat}, then 
$v^\theta \to u$ in $L_t^\infty H_x^1$ as $\theta \to 0$.
The argument is based on energy method together with the known results of the strong space-time bound of energy-critical CGL and NLH, which is similar to the argument of the proof of the inviscid limit of CGL. But we need less regularity than the assmption in \cite{HW} for their inviscid limit of energy-critical NLS.

We now consider the case when $|\theta | \to \frac\pi2$, which is the inviscid limit of \eqref{eq:g-l}. In this case, the dissipative effect becomes weaker, and it is interesting to compare the solutions to energy-critical NLS and CGL. The inviscid limit of CGL has been studied in \cite{BJ,HW, W0, Wu}. In particular, the inviscid limit of \eqref{eq:g-l} in the defocusing case was proved in \cite{HW,W0}. In our second main result, we complete the picture by establishing the inviscid limit of \eqref{eq:g-l} in the focusing case, based on the global well-posedness and scattering of four dimensional focusing energy-critical NLS in \cite{D2019},

\begin{theorem}[Inviscid limit of energy-critical CGL]
\label{thm:f_schrodinger}
  Let $d=4$ and $\mu = - 1 $. Suppose that the initial data $v^\theta_0 \in H^1$ for \eqref{eq:g-l} and $ v_0 \in H^3$ for \eqref{eq:s} satisfy $\left\|v^\theta_0 - v_0 \right\|_{H^1} \le \varepsilon$ for a universal $\varepsilon$ sufficiently small.  We also assume that
    \begin{align}\label{eq1.8v20}
    \max\left\{E \left(v_0^\theta \right), E(v_0 )\right\} < E(W)
    \text{ and }
   \max\left\{\left\|\nabla v_0^\theta \right\|_{L^2}, \|\nabla v_0 \|_{L^2} \right\}< \|\nabla W\|_{L^2}.
   \end{align}
%  \begin{align}\label{eq1.8v20}
%  E \left(v_0^\theta \right), E(v_0 ) < E(W)
% \text{  as well as }
%   \left\|\nabla v^\theta_0  \right\|_{L^2}, \|\nabla v_0 \|_{L^2} < \|\nabla W\|_{L^2}.
%   \end{align}
   Then for any $T\in \left(0,\epsilon|e^{i\theta} -i|^{-1}\right)$, the solutions $v^\theta$ to \eqref{eq:g-l} and $v$ to \eqref{eq:s} satisfy
  \begin{align*}
    \left\|v^\theta - v \right\|_{L^\infty_t H^1_x([0, T] \times \mathbb{R}^d)}
    \le C_1   \left( \left\|v^\theta_0 - v_0 \right\|_{H^1} +  \left|e^{i\theta} - i \right|(T + 1) \right),
  \end{align*}
  where $C_1 $ depends only on $v^\theta_0$ and $v_0$.
\end{theorem}
\begin{remark}
As one may observe from our later proofs, this result can be easily extended to dimension three if the GWP and scattering theory of focusing 3d energy-critical NLS in non-radial case were proven.
\end{remark}

This result reveals that on any $[0, T] \subseteq \mathbb{R}_+$, if the initial datum satisfies the conditions in Theorem \ref{thm:f_schrodinger}, then we have $ v^\theta \to v \text{ in } L_t^\infty H_x^1([0, T] \times \mathbb{R}^d), \text{ as } |\theta | \to \frac\pi2$.
In the proof, the assumption \eqref{eq1.8v20} is necessary to make sure solutions of the energy-critical CGL and NLS in the focusing case are global and have finite space-time norm. The argument is similar to the argument in \cite{BJ,HW, W0, Wu}.

	Our notations are standard. We use $X \lesssim Y$ when $X \le CY$ for some constant $C> 0$ and $X \sim Y$ when $X \lesssim Y \lesssim X$. We write $C=C(a,b,\cdots)$ meaning that $C$ is a constant depending on $a,b,\cdots$. 
	
	For any space-time slab $I \times \mathbb{R}^d$, we use $L_t^q L_x^r(I \times \mathbb{R}^d)$ to denote the space of functions $u: I \times \mathbb{R}^d \to \mathbb{C}$, whose norm satisfies
	$\|u\|_{L_t^q L_x^r (I \times \mathbb{R}^d)} : = \left( \int_I \|u(t) \|_{L_x^r(\mathbb{R}^d) }^q \,\mathrm{d}t \right)^\frac1q < \infty$,
	with the usual modification when $q$ or $r$ are equal to infinity. When $q = r$, we abbreviate $L_t^q L_x^q$ as $L_{t,x}^q$.
	
	We define the $S(I)-$norm of $u$ on the time slab $I$ to be
	\[
\|u\|_{S(I)} : = \left( \int_I \int_{\mathbb{R}^d } |u(t,x)|^\frac{2(d+2)}{d-2}  \,\mathrm{d}x \mathrm{d}t \right)^\frac{d-2}{2(d+2)}.
\]

\section{Energy-dissipation estimates}\label{se2}
In this section, we will first recall some basic theory for the CGL. Then we establish the energy-dissipation estimates for energy-critical CGL and NLH.

\subsection{Basic results for CGL}
In this subsection, we recall the Strichartz estimate of CGL, the variational estimates,  as well as the global well-posedness of energy-critical CGL in \cite{CGZ,HW}, which will be our start point for the limit theory.
\begin{lemma}[Strichartz estimate, \cite{CGZ}]
	\label{lem:strichartz}
	If $\tilde{v} ^\theta $ satisfies
 \begin{align*}
     e^{-i \theta} \pa_t\tilde{v}^\theta  - \Delta \tilde{v}^\theta  = F,
 \end{align*}
 on some interval $[0, T_0 ]$, then  for $t_0 \in [0, T_0]$, we have
 \begin{align*}
 \left\|\tilde{v}^\theta   \right\|_{L_t^q L_x^r }
 \le  \left\| \tilde{v}^\theta (t_0)  \right\|_{L^2} + C \|F\|_{L_t^{\tilde{q}' } L_x^{\tilde{r}'}},
 \end{align*}
 where $ 2\le q, r, \tilde{q},\tilde{r} \le \infty$, $\frac{2}q + \frac{d}r = \frac2{\tilde{q}} + \frac{d}{\tilde{r}} =  \frac{d}2$, with $(q,r,d) , \left(\tilde{q}, \tilde{r}, \tilde{d} \right) \ne (2, \infty, 2)$, and $C$ is a constant depending only on $q,r, \tilde{q}, \tilde{r}, d$.
\end{lemma}

\begin{lemma}[Variational estimates, \cite{CGZ}]\label{le2.1v3}

For $\mu = - 1$, if
\begin{align*}
E \left(v_0^\theta \right) \le ( 1 - \delta_0) E(W),
\left\|\nabla v_0^\theta  \right \|_{L^2} <  \|\nabla W\|_{L^2},
\delta_0 > 0,
\end{align*}
then the solution of \eqref{eq:g-l} satisfies
\begin{align*}
    E \left(v^\theta (t) \right) \sim \|\nabla v^\theta (t) \|_{L_x^2}.
\end{align*}
\end{lemma}

\begin{theorem}[GWP of energy-critical CGL \cite{CGZ,HW}]\label{th2.3v19}
Fix $d=3,4$ and $v_0^\theta \in \dot{H}^1$. Assume one of the following two conditions is satisfied:
\begin{enumerate}
    \item $\mu = 1$,
    \item $\mu = - 1$,  $
    \left\| v_0^\theta  \right\|_{\dot{H}^1} < \| W\|_{\dot{H}^1}$ and $E \left(v_0^\theta \right) < E(W)$.
\end{enumerate}
Then the solution $v^\theta$ to \eqref{eq:g-l} is global and satisfies 
\begin{align}\label{eq2.3v15}
 \left\|v^\theta \right\|_{S(\mathbb{R}_+ ) } \le C=C
 \left( \left\|v_0^\theta \right\|_{\dot{H}^1} \right).
\end{align}

\end{theorem}

\subsection{Energy-dissipation estimates for energy-critical CGL and NLH }

In this subsection, we first give a global energy-dissipation control of \eqref{eq:g-l}, based on Theorem \ref{th2.3v19}, in both defocusing and focusing cases. Then we present the similar results for NLH.

\begin{proposition}[Energy-dissipation estimates for  \eqref{eq:g-l} in the defocusing case] \label{prop:df_gl}
 Fix $d=3, 4$ and $ v^\theta_0 \in \dot{H}^1$. Then the solution $v^\theta$ of \eqref{eq:g-l} satisfies $v^\theta \in C_t^0 \dot{H}^1_x \cap L^2_t \dot{W}_x^{1,\f{2d}{d-2}}$ and
 \begin{align*}
   E \left(v^\theta(t) \right)
   + \cos \theta\int \left|\Delta v^\theta \right|^2 \,\mathrm{d}x \mathrm{d}t
   + \cos \theta  \left\|v^\theta \right\|_{S(\mathbb{R}_+) }^\frac{2(d+2)}{d - 2 }
   \lesssim  E \left(v^\theta_0 \right).
  \end{align*}
  If $v^\theta_0 \in L^2$, then we have
  \begin{align*}
    \left\|v^\theta(t) \right\|_{L^2_x}^2 + \cos \theta \int_0^t
    \left\|\nabla v^\theta( \tau)  \right\|_{L_x^2}^2 \,\mathrm{d} \tau
    \lesssim \left\|v^\theta_0 \right\|_{L^2}^2.
  \end{align*}
\end{proposition}
\begin{proof}
Multiplying \eqref{eq:g-l} by $\overline{v^\theta}$, then integrating over $\mathbb{R}^d$ and integrating by parts, we  arrive at
\begin{align}\label{eq3.4v13}
  \f{d}{dt} \left\|v^\theta \right\|_{L^2_x}^2 + 2\cos\theta \left\|\nabla v^\theta \right\|_{L^2_x}^2
  + \int  \left|v^\theta \right|^{\f{2d}{d-2}} \,\mathrm{d}x =0.
\end{align}
Integrating \eqref{eq3.4v13} from $0$ to $t$, we obtain
\begin{align*}
  \left\|v^\theta \right\|_{L^\infty_tL^2_x}^2 + \cos \theta \left\|v^\theta \right\|_{L^2_t\dot{H}^1_x}^2 + \left\|v^\theta \right\|_{L^2_t L^{\f{2d}{d-2}}_x}^{\f{2d}{d-2}} \lesssim \left\|v^\theta_0 \right\|_{L^2_x}^2.
\end{align*}

 We next turn to the energy formulation. Multiplying \eqref{eq:g-l} by $\overline{v^\theta_t}$ and integrating spatially, the real part of resulted equation is given by
 \begin{align}\label{energy:g_11}
 \begin{aligned}
   &\f{d}{dt} \left( \f12 \left\|\nabla v^\theta \right\|_{L^2_x}^2 + \f{d-2}{2d} \left\|v^\theta \right\|_{L^{\f{2d}{d-2}}_x}^{\f{2d}{d-2}} \right)\\
    &+ \cos\theta\left( \cos \theta\int \text{Re}
   \left(\Delta v^\theta \overline{v^\theta_t} \right) \,\mathrm{d}x
   -\sin\theta \int \text{Im}\left(\Delta v^\theta \overline{v^\theta_t} \right) \,\mathrm{d}x \right)\\
   & - \cos\theta\left(\cos\theta \int \text{Re}\left( f \left(v^\theta \right)
    \overline{v^\theta_t} \right) \,\mathrm{d}x
   - \sin\theta \int \text{Im}
   \left(    f \left(v^\theta \right)
   \overline{v^\theta_t}
   \right) \,\mathrm{d}x \right)= 0.
 \end{aligned}
 \end{align}
Multiplying the complex conjugate of \eqref{eq:g-l} by $\Delta v^\theta$, and then integrating over $\mathbb{R}^d$ and taking the real part of the resulted equation, we find
  \begin{align}\label{energy:g_12}
  \begin{aligned}
    -\cos\theta \int \text{Re}  \left(\Delta v^\theta \overline{v^\theta_t }
    \right) \,\mathrm{d}x  + \sin\theta \int \text{Im}
    \left(\Delta v^\theta \overline{v^\theta_t }\right) \,\mathrm{d}x +
    \left\|\Delta v^\theta \right\|_{L^2_x}^2 \\
    -  \int \text{Re} \left(
    \overline{ f \left(v^\theta \right) }
      \Delta v^\theta    \right) \,\mathrm{d}x =0.
  \end{aligned}
  \end{align}
 Similarly, multiplying the complex conjugate of \eqref{eq:g-l} by $ f  \left(v^\theta \right)   $,
  then integrating the resulted equation over  $\mathbb{R}^d$ and taking the real part, we get
\begin{align}
\label{energy:g_13}
\begin{aligned}
  \cos\theta \int \text{Re} \left( f \left(v^\theta \right)
     \overline{v^\theta_t } \right) \,\mathrm{d}x
    - \sin\theta \int \text{Im} \left( f \left(v^\theta \right)
    \overline{v^\theta_t } \right) \,\mathrm{d}x \\
    - \int \text{Re}\left( f \left(v^\theta \right)
    \overline{\Delta v^\theta} \right) \,\mathrm{d}x
    +  \left\|v^\theta \right\|_{L^{\f{2(d+2)}{d-2}}_x}^{\f{2(d+2)}{d-2}} = 0.
\end{aligned}
  \end{align}

Now we may substitute \eqref{energy:g_12} and \eqref{energy:g_13} into \eqref{energy:g_11} to deduce 
 \begin{align*}
   \begin{aligned}
     \f{d}{dt} E \left(v^\theta \right) + \cos\theta \left\|\Delta v^\theta \right\|_{L^2_x}^2 + \cos\theta \left\|v^\theta \right\|_{L^{\f{2(d+2)}{d-2}}_x}^{\f{2(d+2)}{d-2}}
     -  2\cos\theta \int \text{Re} \left( \overline{ f \left( v^\theta \right) }
      \Delta v^\theta   \right) \,\mathrm{d}x
     = 0.
   \end{aligned}
 \end{align*}
 Furthermore, there holds
 \begin{align*}
   \begin{aligned}
     \int \overline{ f \left( v^\theta \right) }
      \Delta v^\theta \,\mathrm{d}x
     &= - \f d{d-2}\int \left|v^\theta \right|^{\f4{d-2}} \left|\nabla v^\theta \right|^2 \,\mathrm{d}x \\
     &\quad- \f2{d-2} \int \left|v^\theta \right|^{\f4{d-2}-2} \left(\nabla v^\theta \cdot \nabla v^\theta \right)
     \overline{v^\theta}^2 \,\mathrm{d}x ,
   \end{aligned}
 \end{align*}
 which implies
 \begin{align*}
   \int \text{Re}
   \left( \overline{f \left( v^\theta \right) }
    \Delta v^\theta   \right) \,\mathrm{d}x
   \le -\int \left|v^\theta \right|^{\f4{d-2}} \left|\nabla v^\theta \right|^2 \,\mathrm{d}x
   \le 0.
 \end{align*}
 Hence
 \begin{align*}
   \begin{aligned}
     \f{d}{dt} E \left(v^\theta \right) + \cos\theta
     \left\|\Delta v^\theta \right\|_{L^2_x}^2 + \cos\theta \left\|v^\theta \right\|_{L^{\f{2(d+2)}{d-2}}_x}^{\f{2(d+2)}{d-2}}  \le 0,
   \end{aligned}
 \end{align*}
  whence for $t \in \mathbb{R}_+$,
  \begin{align*}
   E\left(v^\theta(t) \right) + \cos\theta \int_0^t \left\|\Delta v^\theta(\tau) \right\|_{L^2_x}^2 \,\mathrm{d}\tau
   + \cos\theta \left\|v^\theta \right\|_{S([0, t])}^\frac{2(d+2)}{d - 2}
   \le E \left(v^\theta_0 \right).
 \end{align*}
\end{proof}

\begin{remark}
The result in Proposition \ref{prop:df_gl} partly has been shown in \cite{HW}. We include this and provide a proof  to make the proposition as self-contained as possible.

\end{remark}

We now establish the global energy-dissipation estimate of \eqref{eq:g-l} in the focusing case.

\begin{proposition}[Energy-dissipation estimates for
\eqref{eq:g-l}
in the focusing case]\label{pr2.6v22}
  Let $d=3, 4$. Suppose that $ v^\theta_0 \in \dot{H}^1$, $E \left(v_0^\theta \right) < E(W)$, and $ \left\|v_0^\theta \right\|_{\dot{H}^1} < \|W\|_{\dot{H}^1}$. Then for a solution $v^\theta$ to \eqref{eq:g-l}, we have $v^\theta \in C_t^0 \dot{H}^1_x \cap L^2_t \dot{W}_x^{1,\f{2d}{d-2}}$ and
  \begin{align*}
    E \left(v^\theta(t) \right) + \cos \theta\int \left|\Delta v^\theta \right|^2 \,\mathrm{d}x \mathrm{d}t
    + \cos \theta \left\| v^\theta \right\|_{S(\mathbb{R}_+)}^\frac{2(d+2)}{d - 2}
    \le C=C\left( \left\|v^\theta_0 \right\|_{\dot{H}^1} \right).
  \end{align*}
  If $v^\theta_0 \in L^2$, then we also have
  \begin{align*}
    \left\|v^\theta(t) \right\|_{L^2_x}
    \le  C=C \left( \left\|v^\theta_0 \right\|_{\dot{H}^1} \right).
  \end{align*}
\end{proposition}
\begin{proof}
 We first derive the energy formulation. Multiplying \eqref{eq:g-l} by $\overline{v^\theta_t } $ and integrating spatially, we find that the real part of the resulted equation is given by
 \begin{align}\label{energy:g_1f1}
 \begin{aligned}
   &\f{d}{dt} \left(\f12 \left\|\nabla v^\theta \right\|_{L^2_x}^2 - \f{d-2}{2d} \left\|v^\theta\right\|_{L^{\f{2d}{d-2}}_x}^{\f{2d}{d-2}} \right) \\
   &+ \cos \theta\int \text{Re} \left(\Delta v^\theta \overline{v^\theta_t }
   \right) \,\mathrm{d}x
   -\sin\theta \int \text{Im} \left(\Delta v^\theta \overline{v^\theta_t }
   \right)\,\mathrm{d}x \\
   & + \cos\theta \int \text{Re}\left( f \left(v^\theta \right) \overline{v^\theta_t } \right) \,\mathrm{d}x
   -\sin\theta \int \text{Im}\left( f \left( v^\theta \right) \overline{v^\theta_t } \right) \,\mathrm{d}x = 0.
 \end{aligned}
 \end{align}
Multiplying the complex conjugate of \eqref{eq:g-l} by $\Delta u$, and then integrating over $\mathbb{R}^d$ and taking the real part of the resulted equation, we obtain
\begin{align}\label{energy:g_1f2}
\begin{aligned}
  -\cos\theta \int \text{Re}
    \left(\Delta v^\theta \overline{v^\theta_t }
    \right) \,\mathrm{d}x
    + \sin\theta \int \text{Im}
    \left(\Delta v^\theta \overline{v^\theta_t }
    \right) \,\mathrm{d}x \\
    + \left\|\Delta v^\theta \right\|_{L^2_x}^2 +  \int \text{Re} \left( \overline{ f \left( v^\theta \right) }
     \Delta v^\theta \right) \,\mathrm{d}x  =0.
\end{aligned}
  \end{align}
 Similarly, multiplying the complex conjugate of \eqref{eq:g-l} by $ f \left( v^\theta \right) $, integrating the resulted equation on $\mathbb{R}^d$ and taking the real part, we find 
\begin{align}\label{energy:g_1f3}
\begin{aligned}
  \cos\theta \int \text{Re} \left( f \left(v^\theta \right)
    \overline{v^\theta_t}
    \right) \,\mathrm{d}x
    - \sin\theta \int \text{Im} \left( f \left(v^\theta \right)
    \overline{v^\theta_t}
    \right)\,\mathrm{d}x \\
    - \int \text{Re} \left(  f \left( v^\theta \right)
    \overline{ \Delta  v^\theta }
    \right) \,\mathrm{d}x
    =  \left\|v^\theta \right\|_{L^{\f{2(d+2)}{d-2}}_x}^{\f{2(d+2)}{d-2}}.
\end{aligned}
  \end{align}
  
Now we substitute \eqref{energy:g_1f2} and \eqref{energy:g_1f3} into \eqref{energy:g_1f1} to get
\begin{align}\label{energy:g_lh1}
   \begin{aligned}
     \f{d}{dt} E \left(v^\theta \right) + \left\|\Delta v^\theta \right\|_{L^2_x}^2 + \left\|v^\theta \right\|_{L^{\f{2(d+2)}{d-2}}_x}^{\f{2(d+2)}{d-2}}
     +  2\int \text{Re}
     \left(\overline{  f \left( v^\theta \right) }
     \Delta v^\theta   \right) \,\mathrm{d}x
     = 0.
   \end{aligned}
 \end{align}
 Furthermore, the inequality
\begin{align*}
   \begin{aligned}
     \int \left|v^\theta \right|^{\f4{d-2}}\Delta v^\theta \overline{v^\theta } \,\mathrm{d}x
     \le \f14  \left\|\Delta v^\theta \right\|_{L^2_x}^2 + 4 \left\|v^\theta \right\|_{L^{\f{2(d+2)}{d-2}}_x}^{\f{2(d+2)}{d-2}}
   \end{aligned}
 \end{align*}
together with \eqref{energy:g_lh1} implies
\begin{align*}
   \begin{aligned}
     \f{d}{dt} E \left(v^\theta \right) + \f12 \left\|\Delta v^\theta \right\|_{L^2_x}^2 \le 2 \left\|v^\theta\right\|_{L^{\f{2(d+2)}{d-2}}_x}^{\f{2(d+2)}{d-2}}.
   \end{aligned}
 \end{align*}
 By \eqref{eq2.3v15}, we have
 \begin{align*}
   E \left(v^\theta(t) \right) + \left\|\Delta v^\theta \right\|_{L_{t,x}^2}
   +   \left\|v^\theta \right\|_{S(\mathbb{R}_+)}
   \le  C \left( \left\|v^\theta_0 \right\|_{\dot{H}^1} \right).
 \end{align*}
Multiplying \eqref{eq:g-l} by $\overline{v^\theta }$, and then  integrating over $\mathbb{R}^d$ and integrating by parts, we arrive at the estimate
\begin{align*}
  \f{d}{dt} \left\|v^\theta \right\|_{L^2_x}^2 + 2\cos\theta \left\|\nabla v^\theta \right\|_{L^2_x}^2 = \int \left|v^\theta \right|^{\f{2d}{d-2}} \,\mathrm{d}x .
\end{align*}
By the Sobolev inequality and Lemma \ref{le2.1v3}, we have
$\left\|v^\theta(t) \right\|_{L^{\f{2d}{d-2}}_x} \le CE \left(v^\theta(t) \right)$.
Therefore, we get
\begin{align*}
  \left\|v^\theta \right\|_{L_t^\infty L^2_x}
  \le  C\left( \left\|v^\theta_0 \right\|_{\dot{H}^1} \right).
\end{align*}
\end{proof}

Before giving the global energy-dissipation estimate of \eqref{eq:h}, we first recall the global well-posedness of \eqref{eq:h}.
\begin{theorem}[GWP of energy-critical NLH, \cite{GR}]\label{th2.4v19}
Fix $d= 3,4$ and $u_0\in \dot{H}^1$. Assume one of the following two conditions is satisfied:
\begin{enumerate}
    \item $\mu = 1$,
    \item $\mu = - 1$, with
\begin{align}
    E(u_0) < E(W), \text{ and }
    \|u_0 \|_{\dot{H}^1} < \|W\|_{\dot{H}^1}.
\end{align}
\end{enumerate}
Then the solution $u$ to \eqref{eq:h}  is global and
\begin{align*}
    \|u\|_{S(\mathbb{R}_+ )  }
    \le C=C \left( \|u_0 \|_{\dot{H}^1} \right).
\end{align*}
\end{theorem}

With the aid of Theorem \ref{th2.4v19}, we obtain the global energy-dissipation estimate of \eqref{eq:h}.
\begin{proposition}[Energy-dissipation estimates \eqref{eq:h} in the defocusing case]\label{prop:df_heat}
  Let $d=3, 4$. Suppose that $ u_0 \in \dot{H}^1$. Then for a solution $u$ to \eqref{eq:h}, we have $u \in C^0_t\dot{H}^1_x \cap L^2_t \dot{W}_x^{1,\f{2d}{d-2}}$ and 
  \begin{align*}
    E(u(t)) + \int |\Delta u|^2 \,\mathrm{d}x \mathrm{d}t
    +  \| u \|_{S(\mathbb{R}_+ )}^\frac{2(d+2)}{d - 2 }
    \lesssim E(u_0).
  \end{align*}
  Furthermore, if $u_0 \in L^2$, then we also have
  \begin{align*}
    \|u(t)\|_{L^2_x}^2 + \int_0^t \|\nabla u(\tau) \|_{L^2}^2 \,\mathrm{d}\tau  \lesssim \|u_0\|_{L^2}^2.
  \end{align*}
\end{proposition}

\begin{proposition}[Energy-dissipation estimates for  \eqref{eq:h} in the focusing case]\label{prop:f_heat}
  Let $d=3, 4$. Suppose that $u_0 \in \dot{H}^1$ satisfies
  \begin{align*}
 E(u_0) < E(W)\quad \text{and}\quad \|u_0 \|_{\dot{H}^1} < \|W\|_{\dot{H}^1}.
  \end{align*}
  Then for a solution $u$ to \eqref{eq:h}, we have
  \begin{align*}
    E(u(t)) + \int|\Delta u |^2 \,\mathrm{d}x \mathrm{d}t
    +  \|u \|_{S( \mathbb{R}_+ )}^\frac{2(d+2)}{d - 2}
    \le C=C \left(\|u_0\|_{\dot{H}^1} \right).
  \end{align*}
  If $u_0 \in L^2$, then we also have
  \begin{align*}
    \|u(t)\|_{L^2_x}
    \le C=C \left(\|u_0\|_{\dot{H}^1} \right).
  \end{align*}
\end{proposition}

\begin{proof}[Proof of Propositions \ref{prop:df_heat} and \ref{prop:f_heat}]
We can apply the same argument as that used in the proof of Propositions \ref{prop:df_gl} or \ref{pr2.6v22}. The only difference is that we use Theorem \ref{th2.4v19} instead of Theorem \ref{th2.3v19}.
\end{proof}

\section{Proof of the zero-dispersion limit theory}\label{se3}

In this section, we show the zero-dispersion limit of \eqref{eq:g-l} via the energy method, where one of the ingredients is the higher regularity of solutions to limiting equation.

\subsection{Higher regularity for energy-critical nonlinear heat equation}
In this subsection, we show the solution of energy-critical NLH has higher regularity when the initial data lies in $H^2$.
\begin{proposition}[Higher regularity for nonlinear heat equation]\label{prop:higher_heat}
  Suppose that $u$ is a solution to  \eqref{eq:h} and that $u_0 \in H^2$. When $\mu = - 1$, we additionally assume
  \begin{align*}
      E(u_0) <E(W)\quad \text{and}\quad \|u_0 \|_{\dot{H}^1} < \|W\|_{\dot{H}^1}.
  \end{align*}
  Then we have
  \begin{align}\label{eq3.16v15}
    \|u\|_{L^\infty_t H^2_x} + \|u\|_{L_t^2 \dot{H}^3_x}
    \le C=C\left(\|u_0\|_{H^2} \right).
  \end{align}
\end{proposition}

\begin{proof}
  From Theorem \ref{th2.4v19}, we know that
  \begin{align*}
    \|u\|_{ S(\mathbb{R}_+) } + \|\nabla u\|_{L^{\f{2(d+2)}{d-2}}_t L^{\f{2d(d+2)}{d^2+4}}_x (\mathbb{R}_+ \times \mathbb{R}^d) }
    \le C(\|u_0\|_{\dot{H}^1}).
  \end{align*}
For any sufficiently small $\delta_0 >0$, we can divide $\mathbb{R_{+}}$ into a finite number of disjoint intervals, $\mathbb{R_{+}} : = \bigcup_{j=1}^N I_j$, where $I_j = [ t_{j-1}, t_j]$ and $N = C( \|u_0 \|_{\dot{H}^1} )$, so that
  \begin{align*}
    \|u\|_{S(I_j) } + \|\nabla u\|_{L^{\f{2(d+2)}{d-2}}_t L_x^{\f{2d(d+2)}{d^2+4}}
    (I_j \times \mathbb{R}^d) }
    \le \delta_0 , \quad j = 1, \ldots, N.
  \end{align*}
By Lemma \ref{lem:strichartz}, we have
\begin{align}\label{est:h_heat1}
    \begin{aligned}
      \left\|\nabla^2 u \right\| _{L^{\infty}_t L^2_x \bigcap L^2_t L^{\f{2d}{d-2}}_x (I_j \times \mathbb{R}^d)}
      & \le \left\| \nabla^2 u(t_{j-1}) \right\|_{L^2} +
      C
      \left\| |u|^{\f 4{d-2}} \nabla^2 u \right\|_{L^{\f{2d+4}{d+4}}_{t,x}(I_j \times \mathbb{R}^d)} \\
      &\quad+ C
      \left\| |u|^{\f{6-d}{d-2}} |\nabla u|^2 \right\|_{ L^2_t L^{\f{2d}{d-2}}_x (I_j \times \mathbb{R}^d) }.
    \end{aligned}
  \end{align}
 Then H\"older's inequality allows us to estimate
\begin{align}\label{est:h_heat2}
    \begin{aligned}
      \left\| |u|^{\f 4{d-2}} \nabla^2 u \right\|_{L^{\f{2d+4}{d+4}}_{t,x}(I_j \times \mathbb{R}^d)}
      &\le \|u\|_{S(I_j) }^{\f4{d-2}}
      \left\|\nabla^2 u \right\|_{L^{\f{2d+4}d}_{t,x}(I_j \times \mathbb{R}^d)} \\
    &\le \delta_0^{\f 4{d-2}}
      \left\|\nabla^2 u \right\|_{L^{\infty}_t L^2_x \bigcap L^2_t L^{\f{2d}{d-2}}_x (I_j \times \mathbb{R}^d)}.
    \end{aligned}
  \end{align}
  By the H\"older inequality and the Sobolev inequality, we have
\begin{align}\label{est:h_heat3}
    \begin{aligned}
      &\left\| |u|^{\f{6-d}{d-2}} |\nabla u|^2 \right\|_{ L^2_t L^{\f{2d}{d-2}}_x (I_j \times \mathbb{R}^d) }\\
      & \le
      \|u\|_{ L^{\f{2d+4}{d-2}}_{t,x} (I_j \times \mathbb{R}^d) }^{\f{6-d}{d-2}} \|\nabla u\|_{ L^{\f{2d+4}{d-2}}_tL^{\f{2d(d+2)}{d^2+4}}_x(I_j \times \mathbb{R}^d)} \left\|\nabla^2 u \right\|_{ L^{\f{2d+4}{d-2}}_tL^{\f{2d(d+2)}{d^2+4}}_x(I_j \times \mathbb{R}^d)}\\
      &\le \delta_0^{\f 4{d-2}}
      \left\|\nabla^2 u \right\|_{L^{\infty}_t L^2_x \bigcap L^2_t L^{\f{2d}{d-2}}_x (I_j \times \mathbb{R}^d)}.
    \end{aligned}
  \end{align}
  Hence, combining \eqref{est:h_heat1}, \eqref{est:h_heat2}, and \eqref{est:h_heat3}, we have
\begin{align}\label{est:h_heat4}
    \left\|\nabla^2 u \right\|_{L^{\infty}_t L^2_x \bigcap L^2_t L^{\f{2d}{d-2}}_x (I_j \times \mathbb{R}^d)}
    \le \left\| \nabla^2 u(t_{j-1}) \right\|_{L^2}
    + 2 C \delta_0^{\f 4{d-2}}
    \left\|\nabla^2 u \right\|_{L_t^\infty L_x^2 \cap L_t^2 L_x^\frac{2d}{d - 2}
    (I_j \times \mathbb{R}^d)}.
  \end{align}
When $j=1$,  $u(t_{j-1}) = u_0$ and so \eqref{est:h_heat4} yields
  \begin{align*}
    \left\|\nabla^2 u \right\|_{L^{\infty}_t L^2_x \bigcap L^2_t L^{\f{2d}{d-2}}_x (I_1 \times \mathbb{R}^d)} \le 2 \left\| \nabla^2 u_0 \right\|_{L^2},
  \end{align*}
  which in particular implies $ \left\| \nabla^2 u(t_1) \right\|_{L^2} \le 2
  \left\| \nabla^2 u_0 \right\|_{L^2}$.

  Therefore, repeating this procedure finitely many times, we may deduce
  \begin{align*}
    \left\|\nabla^2 u \right\|_{L^{\infty}_t L^2_x \bigcap L^2_t L^{\f{2d}{d-2}}_x (I_j \times \mathbb{R}^d)} \le 2^j \left\|\nabla^2u_0 \right\|_{L^2}, \quad j=1, \ldots, N.
  \end{align*}
  Hence,
\begin{align}\label{est:h_heat11}
    \begin{aligned}
      \left\|\nabla^2 u \right\| _{L^{\infty}_t L^2_x ( \mathbb{R_{+}} \times \mathbb{R}^d)} \le
      \sum_{j=1}^N 2^j  \left\|\nabla^2u_0 \right\|_{L^2} \le C \left(\|u_0\|_{\dot{H}^1} \right) \|u_0\|_{\dot{H}^2}.
    \end{aligned}
  \end{align}
 
Next, we estimate $\|u\|_{L^2_t\dot{H}^3_x}$. Taking $\Delta$ on both sides of \eqref{eq:h}, and then multiplying by $\overline{ \Delta u}$, we arrive at the energy formula
  \begin{align}\label{energy:higher1}
  \begin{aligned}
    \frac{1}{2} \frac{d}{dt} \int_{\mathbb{R}^d} |\Delta u|^2 \,\mathrm{d}x+ \int_{\mathbb{R}^d} |\nabla\Delta u|^2 \,\mathrm{d}x & \lesssim \int_{\mathbb{R}^d}
    |\Delta u|^2|u|^{\frac{4}{d-2}} + |u|^{\frac{6-d}{d-2}}|\nabla u|^2 |\Delta u| \,\mathrm{d}x .
    \end{aligned}
  \end{align}
  The first integral on the right-hand side of \eqref{energy:higher1} can be bounded by
  \begin{align}\label{energy:higher2}
    \int_{\mathbb{R}^d} |\Delta u|^2|u|^{\frac{4}{d-2}} \,\mathrm{d}x
    \le  \|\Delta u\|_{L^{\frac{2(d+2)}{d}}_x}^2 \|u\|_{L^{\frac{2(d+2)}{d-2}}_x}^{\frac{4}{d-2}}.
  \end{align}
  The second integral on the right-hand side of \eqref{energy:higher1} can be estimated as follows
  \begin{align}
  \label{energy:higher3}
    \int_{\mathbb{R}^d}  |u|^{\frac{6-d}{d-2}}|\nabla u|^2 |\Delta u| \,\mathrm{d}x  \le
    \begin{cases}
     \|u\|_{L^\infty}^3 \|\nabla u\|_{L^6}^2\|\Delta u\|_{L^2} \lesssim \|\nabla^2 u\|_{L^2}^6, & d = 3, \\
    \|u\|_{L^4} \|\nabla u\|_{L^4}^2\|\Delta u\|_{L^4} \lesssim \|\nabla^2 u\|_{L^2}^3 \|\nabla \Delta u\|_{L^2}, & d = 4.
    \end{cases}
  \end{align}
  Integrating \eqref{energy:higher1} on $I_j$, and using \eqref{energy:higher2} and \eqref{energy:higher3}, we obtain
  \begin{align}\label{dissipation:higher1}
  \begin{aligned}
      \|\nabla\Delta u\|_{L^2_{t,x }
      (I_j\times \mathbb{R}^d)}^2
      & \lesssim \|\Delta u(t_{j-1})\|_{L^2}^2 + \|\Delta u\|_{L^{\frac{2(d+2)}{d}}_{t,x}(I_j\times \mathbb{R}^d)}^{\frac{d}{d+2}}
      \|u\|_{S (I_j )}^{\frac{2}{d+2}} \\
   &\quad+ \|\Delta u\|_{L^2_{t,x }(I_j\times \mathbb{R}^d)}^2 \|\Delta u\|_{L^\infty_t L^2_x(I_j\times \mathbb{R}^d)}^4.
    \end{aligned}
  \end{align}
Thus, repeating this procedure from $j=1, \ldots, N$, and then combing it with the bound of $\|\Delta u\|_{L^2_{t,x}}$ in Propositions \ref{prop:df_heat} and \ref{prop:f_heat}, and also the bound for $\|\Delta u\|_{L^\infty_t L^2_x}$ in \eqref{est:h_heat11}, we eventually obtain 
  \begin{align}\label{dissipation:higher2}
  \begin{aligned}
      \|\nabla\Delta u\|_{L^2_{t,x } (\mathbb{R}_+\times \mathbb{R}^d)}
     \le   C(\|u_0\|_{H^1})
     \left (\|u_0\|_{\dot{H}^2} +\|u_0\|_{\dot{H}^2}^3 \right).
    \end{aligned}
  \end{align}

  \end{proof}

\subsection{Proof of Theorem \ref{thm:df_heat}}
In this subsection, we complete the proof of Theorem \ref{thm:df_heat}.

Set $w =  v^\theta - u $. It is clear that $w$ satisfies
\begin{align}\label{eq:difference1}
  \begin{cases}
e^{-i\theta}w_t -  \Delta w = \mu \left( f(v^\theta) - f(u) \right)
-  \left(e^{-i\theta}-1 \right) \left(\Delta u - \mu f(u) \right), \\
w(0,x) = w_0(x) =  v^\theta_0(x) - u_0(x).
\end{cases}
\end{align}
For simplicity, we write $F(v^\theta, u)=f \left(v^\theta \right) - f(u)-f(w)$. By direct computation, we know that
when $d=3$,
\begin{align}\label{eq:df_non2}
\begin{aligned}
  F \left(v^\theta, u \right)
  &=  \left(\bar{w}u + w\bar{u} \right)^2v^\theta + 2|w|^2|u|^2v^\theta + 2 \left(|w|^2+|u|^2 \right) \left(\bar{w}u + w\overline{v^\theta }
  \right) v^\theta + |u|^4 v^\theta\\
  &\quad -\left|v^\theta \right|^4u - 2|w|^2 \left|v^\theta \right|^2u - \left(\bar{w }
  v^\theta + w\overline{v^\theta }
  \right)^2u \\
  &\quad+ 2 \left(|w|^2+ \left|v^\theta \right|^2 \right) \left(\bar{w}v^\theta + w\overline{v^\theta }
  \right)u,
\end{aligned}
\end{align}
and
when $d=4$,
\begin{align}\label{eq:df_non1}
   F\left(v^\theta, u \right) = w \left(\overline{v^\theta }
   u + \bar{u} v^\theta \right) +2\bar{w}v^\theta u.
\end{align}

\begin{lemma}\label{prop:df_diff}
  Let $d = 3, 4$. Suppose that $w$ satisfies \eqref{eq:difference1}. There exists a small $\varepsilon > 0 $ such that if $\|w_0\|_{H^1} \le \varepsilon$, then for each $T>0$, we have
\begin{align}\label{est:df_w}
    \|w\|_{S([0, T]) } \le C=C\left(T,\|w_0\|_{\dot{H}^1}, \|u_0\|_{\dot{H}^2} \right).
  \end{align}
\end{lemma}
\begin{proof}
We can divide $[0, T]$ into a finite number $N = C( \|u_0 \|_{\dot{H}^1} ) $ of disjoint intervals  $(I_j)_{j=1}^N$ such that the solution $u$ of \eqref{eq:h}
satisfies
\begin{align}\label{est:df_w1}
  \|u\|_{S(I_j) }+ \|\nabla u\|_{L^{\f{2d+4}{d-2}}_t L^{\f{2d(d+2)}{d^2+4}}_x(I_j \times \mathbb{R}^d)} \le \delta.
\end{align}
Then we may apply the Strichartz estimate to
\eqref{eq:difference1} to obtain 
\begin{align}\label{est:df_w2}
\begin{aligned}
    &\|\nabla w\|_{L^\infty_t L^2_x(I_j \times \mathbb{R}^d)} + \|\nabla w\|_{L^{\f{2(d+2)}{d-2}}_t L^{\f{2d(d+2)}{d^2+4}}_x(I_j \times \mathbb{R}^d)} \\
    & \le \|w(t_{j-1})\|_{\dot{H}^1}+ C |e^{i\theta} -1|
    \left( \left\|\nabla\Delta u \right\|_{L^1_tL^2_x(I_j \times \mathbb{R}^d)} +   \left\|\nabla    f(u)  \right\|_{L^2_t L_x^{\f{2d}{d+2}}(I_j \times \mathbb{R}^d)} \right)   \\
     &\quad + C\left\|\nabla  \left( f \left(v^\theta \right) - f(u)     \right) \right\|_{L^2_t L_x^{\f{2d}{d+2}}(I_j \times \mathbb{R}^d)}.
\end{aligned}
\end{align}
It is clear that
\begin{align}\label{est:df_w3}
  \left\|\nabla\Delta u \right\|_{L^1_tL^2_x(I_j \times \mathbb{R}^d)} \le |I_j|^{1/2} \|u\|_{L^2_t\dot{H}^3_x(I_j \times \mathbb{R}^d)}.
\end{align}
Then H\"older's inequality and \eqref{est:df_w1} give
\begin{align}\label{est:df_w4}
    \begin{aligned}
      \left\|\nabla f(u)
      \right\|_{L^2_t L_x^{\f{2d}{d+2}}(I_j \times \mathbb{R}^d)}
   \le \|u\|_{S(I_j)  }^{\f4{d-2}} \|\nabla u\|_{L^{\f{2d+4}{d-2}}_t L^{\f{2d(d+2)}{d^2+4}}_x(I_j \times \mathbb{R}^d)} \lesssim \delta^{\f{d+2}{d-2}}.
    \end{aligned}
  \end{align}
 By
  \eqref{eq:df_non2} and \eqref{eq:df_non1}, we have
  \begin{align*}
      \left| \nabla \left( f \left( v^\theta \right) - f(u)       \right) \right|
      \lesssim |w+u|^{\f4{d-2}} |\nabla w|
      + \left( |w|^{\f{6-d}{d-2}} + |u|^{\f{6-d}{d-2}}\right) |w||\nabla u|,
  \end{align*}
which implies by H\"older's inequality, \eqref{est:df_w1} and Sobolev's inequality that
  \begin{align}\label{est:df_w5}
    \begin{aligned}
      & \left\|\nabla
      \left( f \left(v^\theta \right) - f(u)
       \right) \right\|_{L^2_t L_x^{\f{2d}{d+2}}(I_j \times \mathbb{R}^d)} \\
     & \lesssim \left( \|\nabla w\|_{L^{\f{2d+4}{d-2}}_t L^{\f{2d(d+2)}{d^2+4}}_x(I_j \times \mathbb{R}^d)}^{\f4{d-2}} + \delta^{\f4{d-2}} \right) \|\nabla w\|_{L^{\f{2d+4}{d-2}}_t L^{\f{2d(d+2)}{d^2+4}}_x(I_j \times \mathbb{R}^d)} \\
      & \quad + \delta \left( \|\nabla w\|_{L^{\f{2d+4}{d-2}}_t L^{\f{2d(d+2)}{d^2+4}}_x(I_j \times \mathbb{R}^d)}^{\f{6-d}{d-2}} + \delta^{\f{6-d}{d-2}} \right) \|\nabla w\|_{L^{\f{2d+4}{d-2}}_t L^{\f{2d(d+2)}{d^2+4}}_x(I_j \times \mathbb{R}^d)} \\
      &\lesssim \|\nabla w\|_{L^{\f{2d+4}{d-2}}_t L^{\f{2d(d+2)}{d^2+4}}_x(I_j \times \mathbb{R}^d)}^{\f{d+2}{d-2}} + \delta^{\f4{d-2}}  \|\nabla w\|_{L^{\f{2d+4}{d-2}}_t L^{\f{2d(d+2)}{d^2+4}}_x(I_j \times \mathbb{R}^d)}\\
        &\quad+ \delta \|\nabla w\|_{L^{\f{2d+4}{d-2}}_t L^{\f{2d(d+2)}{d^2+4}}_x(I_j \times \mathbb{R}^d)}^{\f4{d-2}}.
    \end{aligned}
  \end{align}
  When $j=1$, $w(t_0) = w_0$. Combining the estimates \eqref{est:df_w2}--\eqref{est:df_w5} with Proposition \ref{prop:higher_heat}, we infer that  for a universal constant $C$, it holds 
\begin{align}\label{est:df_w6}
    \begin{aligned}
      &\|\nabla w\|_{L^\infty_t L^2_x(I_1 \times \mathbb{R}^d)} + \|\nabla w\|_{L^{\f{2(d+2)}{d-2}}_t L^{\f{2d(d+2)}{d^2+4}}_x(I_1 \times \mathbb{R}^d)} \\
    & \le  \|w_0\|_{\dot{H}^1}+ C|e^{i\theta} -1|
    \left( T^{1/2} C \left(\|u_0\|_{\dot{H}^1} \right)\|u_0\|_{\dot{H}^2}+ \delta^{\f{d+2}{d-2}} \right) \\
    &\quad + C\|\nabla w\|_{L^{\f{2d+4}{d-2}}_t L^{\f{2d(d+2)}{d^2+4}}_x(I_1 \times \mathbb{R}^d)}^{\f{d+2}{d-2}} + C\delta^{\f4{d-2}}
    \|\nabla w\|_{L^{\f{2d+4}{d-2}}_t L^{\f{2d(d+2)}{d^2+4}}_x(I_1 \times \mathbb{R}^d)}\\
      &\quad + C\delta \|\nabla w\|_{L^{\f{2d+4}{d-2}}_t L^{\f{2d(d+2)}{d^2+4}}_x(I_1 \times \mathbb{R}^d)}^{\f4{d-2}}.
    \end{aligned}
  \end{align}

  We choose a small positive constant $\varepsilon$ such that $\|w_0\|_{\dot{H}^1} \le \f\varepsilon2$ and $|e^{i\theta}-1| T^{1/2} \le \f\varepsilon4$. By a standard continuity method, \eqref{est:df_w6} is reduced to
\begin{align}\label{est:df_w7}
    \|\nabla w\|_{L^\infty_t L^2_x(I_1 \times \mathbb{R}^d)} + \|\nabla w\|_{L^{\f{2(d+2)}{d-2}}_t L^{\f{2d(d+2)}{d^2+4}}_x(I_1 \times \mathbb{R}^d)} \le 2C\varepsilon.
  \end{align}
  In particular, \eqref{est:df_w7} gives the bound $\|\nabla w (t_1)\|_{L^2} \le 2C\varepsilon$. Hence by a finite induction argument corresponding to the finite partition of $[0, T]$, we have
  \begin{align*}
    \|\nabla w\|_{L^\infty_t L^2_x(I_N \times \mathbb{R}^d)} + \|\nabla w\|_{L^{\f{2(d+2)}{d-2}}_t L^{\f{2d(d+2)}{d^2+4}}_x(I_N \times \mathbb{R}^d)} \le (4C)^N\varepsilon.
  \end{align*}
Choosing $\epsilon$ small enough such that $(4C)^N\varepsilon \le \delta$, we thus derive
  \begin{align*}
    \|\nabla w\|_{L^\infty_t L^2_x([0,T] \times \mathbb{R}^d)} + \|\nabla w\|_{L^{\f{2(d+2)}{d-2}}_t L^{\f{2d(d+2)}{d^2+4}}_x([0,T] \times \mathbb{R}^d)} \le C    \left(T,\|w_0\|_{\dot{H}^1}, \|u_0\|_{\dot{H}^2}\right),
  \end{align*}
which along with Sobolev's inequality  implies \eqref{est:df_w}.
\end{proof}

\begin{proof}[Proof of Theorem \ref{thm:df_heat}.]

  From the Strichartz estimate and H\"older's  inequality, we may estimate
  \begin{align}\label{est:main1_1}
    \begin{aligned}
      \|w\|_{L^\infty_t L^2_x \cap L^2_tL^{\f{2d}{d-2}}_x([0,T] \times \mathbb{R}^d)}
      &\le
      \|w_0\|_{L^2} + C |e^{i\theta} -1|
      \bigg( \|\Delta u\|_{L^1_tL^2_x([0,T] \times \mathbb{R}^d)} \\
      &\quad+
C \left\| f(u) \right\|_{L_{t,x}^{\f{2d+4}{d+4}}([0,T] \times \mathbb{R}^d)} \bigg)\\
      &\quad
      + C  \left\|\nabla \left( f \left(v^\theta \right) - f(u)
      \right) \right\|_{L_{t,x}^{\f{2d+4}{d+4}}([0,T] \times \mathbb{R}^d)} \\
      &\le  \|w_0\|_{L^2} + C  |e^{i\theta} -1| |I|\|\Delta u\|_{L^\infty_tL^2_x([0,T] \times \mathbb{R}^d)} \\
      &\quad + C \left|e^{i\theta} -1 \right|
      \|u\|_{ S([0, T]) }^{\f4{d-2}} \|u\|_{L^{\f{2d+4}d}_{t,x}([0,T] \times \mathbb{R}^d)}\\
      &\quad + C \left( \|w\|_{
      S([0, T])  }^{\f4{d-2}} + \|u\|_{
      S([0, T]) }^{\f4{d-2}} \right) \| w\|_{L^{\f{2(d+2)}{d}}_{t,x}([0,T] \times \mathbb{R}^d)}.
    \end{aligned}
  \end{align}
 Lemma \ref{prop:df_diff} and Theorem \ref{th2.4v19} allow us to divide the interval $[0, T]$ into finite disjoint intervals $I_j$, $j=1, \ldots, N$, such that
\begin{align}\label{est:main1_2}
    C\left( \|w\|_{
    S(I_j) }^{\f4{d-2}} + \|u\|_{
    S(I_j)  }^{\f4{d-2}} \right) \le \f12, \quad j=1, \ldots, N,
  \end{align}
  where $C$ is the same as in \eqref{est:main1_1}. Then \eqref{est:main1_1} can be used to derive
\begin{align}\label{est:main1_3}
    \begin{aligned}
      \|w\|_{L^\infty_t L^2_x \cap L^2_tL^{\f{2d}{d-2}}_x(I_j \times \mathbb{R}^d)}
     & \le \|w(t_{j-1})\|_{L^2} + C |e^{i\theta} -1||I_j|^{1/2} \|\Delta u\|_{L^2_{t , x}
      (I_j \times \mathbb{R}^d)} \\
      & \quad + \f12 |e^{i\theta} -1| \|u\|_{L^{\f{2d+4}d}_{t,x}(I_j \times \mathbb{R}^d)} + \f12 \| w\|_{L^{\f{2(d+2)}{d}}_{t,x}(I_j \times \mathbb{R}^d)}.
    \end{aligned}
  \end{align}
This together with the interpolation
  \begin{align}
    \| w\|_{L^{\f{2(d+2)}{d}}_{t,x}(I_j \times \mathbb{R}^d)} \le \|w\|_{L^\infty_t L^2_x \cap L^2_tL^{\f{2d}{d-2}}_x(I_j \times \mathbb{R}^d)}
  \end{align}
implies
\begin{align}\label{est:main1_4}
    \begin{aligned}
      \|w\|_{L^\infty_t L^2_x \cap L^2_tL^{\f{2d}{d-2}}_x(I_j \times \mathbb{R}^d)} &\le 2 \|w\|_{L^\infty_t L^2_x \cap L^2_tL^{\f{2d}{d-2}}_x(I_{j-1} \times \mathbb{R}^d)} + C |e^{i\theta} -1| \|u\|_{L^{\f{2d+4}d}_{t,x}(I_j \times \mathbb{R}^d)}\\
      & \quad + 2C \left|e^{i\theta} -1 \right||I_j|^{1/2} \|\Delta u\|_{L^2_{t, x} (I_j \times \mathbb{R}^d)}.
    \end{aligned}
  \end{align}
  Hence, by iteration, for $j= 1, \ldots, N$, we have
\begin{align}\label{est:main1_5}
    \begin{aligned}
      \|w\|_{L^\infty_t L^2_x \cap L^2_tL^{\f{2d}{d-2}}_x(I_j \times \mathbb{R}^d)}
      & \le 2^j \|w_0\|_{L^2} + C \left|e^{i\theta} -1 \right| \sum_{k=1}^j 2^k |I_{j-k+1}|^{1/2} \|\Delta u\|_{L^2_{t, x} (I_{j-k+1} \times \mathbb{R}^d)} \\
      & \quad + C  |e^{i\theta} -1| \sum_{k=1}^j 2^{k - 1}  \|u\|_{L^{\f{2d+4}d}_{t,x}(I_{j-k+1} \times \mathbb{R}^d)}.
    \end{aligned}
  \end{align}
  Summing over $j$ and using the estimate
  \begin{align}
    \sup_{k\le j} \|u\|_{L^{\f{2d+4}d}_{t,x}(I_{j-k+1} \times \mathbb{R}^d)} \le C(E(u_0)) \|u_0\|_{\dot{H}^2},
  \end{align}
   which follows from the Strichartz estimate and \eqref{est:main1_2},
  we conclude 
  \begin{align}\label{eq3.33v21}
    \begin{aligned}
      \|w\|_{L^\infty_t L^2_x ([0, T] \times \mathbb{R}^d)}
      &\le  C 2^{N+2} \big(\|w_0\|_{L^2} + |e^{i\theta} -1|T^{1/2} \|\Delta u\|_{L^2_{t, x} ([0, T] \times \mathbb{R}^d)}\\
      &\quad + |e^{i\theta} -1| \|u_0\|_{\dot{H}^2} \big),
    \end{aligned}
  \end{align}
  where $N$ is independent of $\theta$.
  This together with \eqref{est:h_heat11} implies the result for $L^2$. The approximated result for $\dot{H}^1$ is similar to the $L^2$ case and can be essentially derived by Lemma \ref{prop:df_diff}, so we omit the details here.
\end{proof}

\section{ Proof of the invisicid limit theory}
In this section, we show the inviscid limit of \eqref{eq:g-l}, that is Theorem \ref{thm:f_schrodinger}. The proof relies on the following global well-posedness and scattering of energy critical NLS.

\begin{theorem}[GWP \& scattering of 4 dimensional focusing energy-critical NLS, \cite{D2019}]
Fix $d  = 4$, $\mu = - 1$ and $v_0 \in \dot{H}^1.$ If
\begin{align*}
    E(v_0) <E(W)\quad\text{and} \quad \|v_0 \|_{\dot{H}^1} < \|W\|_{\dot{H}^1},
\end{align*}
then the solution $v$ to \eqref{eq:s} is global and satisfies
\begin{align*}
\|v\|_{S(\mathbb{R} ) } \le C=C \left(\|v_0 \|_{\dot{H}^1} \right).
\end{align*}

\end{theorem}

We now give the higher regularity of NLS.
\subsection{Higher regularity of NLS}
\begin{proposition}[Higher regularity for NLS] \label{prop:higher_nls}
  Suppose that $v$ is a solution to  \eqref{eq:s} and that $v_0 \in H^3$.
  In addition, we assume
  \begin{align*}
   E(v_0) < E(W)\quad\text{and} \quad\|v_0 \|_{\dot{H}^1} < \|W\|_{\dot{H}^1}.
  \end{align*}
  Then for $k=2, 3$, we have
  \begin{align}\label{eq3.16v16}
    \|v\|_{L^\infty_t H^k_x}
    \le C=C\left(\|v_0\|_{H^k} \right).
  \end{align}
\end{proposition}

\begin{proof}
The proof for the case $k=2$ is the same as that used in Proposition \ref{prop:higher_heat} and thus we only consider the case when $v_0 \in \dot{H}^3$.
Divide $[0, T]$ into a finite number $N$ of disjoint intervals $I_j$ such that the solution $v$ of \eqref{eq:s} satisfies
\begin{align}\label{est:df_s}
  \|v\|_{S(I_j) }+ \|\nabla v \|_{L^{\f{2d+4}{d-2}}_t L^{\f{2d(d+2)}{d^2+4}}_x(I_j \times \mathbb{R}^d)} \le \delta_0.
\end{align}
From the Duhamel formula
\begin{align}\label{eq:duhamel2}
    v(t)  =  e^{it \Delta}
    v_0 - \mu \int_0^t
    e^{i(t- s) \Delta} f(v(s))    \,\mathrm{d}s,
  \end{align}
we know
  \begin{align*}
    \nabla v_t = e^{it \Delta }
    \nabla \Delta v_0 -
    e^{it \Delta} \nabla f(v_0)
    - \mu \int_0^t  e^{i(t- s) \Delta}
    \nabla \pa_s f(v(s))    \,\mathrm{d}s.
  \end{align*}
 Applying the Strichartz estimate on  each $I_j$, we obtain 
\begin{align}\label{est:h_s5}
  \begin{aligned}
    \|\nabla v_t\|_{L^{\infty}_t L^2_x \bigcap L^2_t L^{\f{2d}{d-2}}_x (I_j \times \mathbb{R}^d)}
    & \le
    \left\|v(t_{j-1}) \right\|_{\dot{H}^3} + C \left\|\nabla
    f(u(t_{j-1}) )  \right\|_{L^2}
    \\
    & \quad+ C
    \left\| |v|^{\f 4{d-2}} \nabla v_t \right\|_{L^{\f{2d+4}{d+4}}_{t,x}(I_j \times \mathbb{R}^d)} \\
    & \quad + C \left\| |v|^{\f{6-d}{d-2}} \nabla v v_t \right\|_{ L^2_t L^{\f{2d}{d-2}}_x (I_j \times \mathbb{R}^d) }.
  \end{aligned}
  \end{align}
  By H\"older's inequality, Sobolev's inequality and \eqref{est:df_s}, we deduce
\begin{align}\label{est:h_s6}
  \begin{aligned}
    \left\| |v|^{\f 4{d-2}} \nabla v_t \right\|_{L^{\f{2d+4}{d+4}}_{t,x}(I_j \times \mathbb{R}^d)}
    & \le \|v\|_{S(I_j) }^{\f4{d-2}} \|\nabla v_t\|_{L^{\f{2d+4}d}_{t,x}(I_j \times \mathbb{R}^d)} \\
&\le \delta_0^{\f 4{d-2}} \|\nabla v_t\|_{L^{\infty}_t L^2_x \bigcap L^2_t L^{\f{2d}{d-2}}_x (I_j \times \mathbb{R}^d)},
  \end{aligned}
  \end{align}
  and
\begin{align}\label{est:h_s7}
    \begin{aligned}
      \left\| |v|^{\f{6-d}{d-2}} \nabla v v_t \right\|_{ L^2_t L^{\f{2d}{d-2}}_x (I_j \times \mathbb{R}^d) }
    &\le
    \|v\|_{S(I_j) }^{\f{6-d}{d-2}} \|\nabla v_t\|_{
    L^{\f{2d+4}{d-2}}_tL^{\f{2d(d+2)}{d^2+4}}_x(I_j \times \mathbb{R}^d)
    }\\
    &\quad \times\|\nabla v\|_{ L^{\f{2d+4}{d-2}}_tL^{\f{2d(d+2)}{d^2+4}}_x(I_j \times \mathbb{R}^d)}\\
      &\le \delta_0^{\f 4{d-2}} \|\nabla v_t\|_{L^{\infty}_t L^2_x \bigcap L^2_t L^{\f{2d}{d-2}}_x (I_j \times \mathbb{R}^d)}.
    \end{aligned}
  \end{align}
 For the second term on the right hand side of  \eqref{est:h_s5},  we may apply Sobolev's inequality and interpolation to get  \begin{align}\label{est:h_s8}
    \begin{aligned}
      \left\|\nabla f
      \left(v \left(t_{j-1} \right)
 \right)\right\|_{L^2}
      &\lesssim \|v(t_{j-1})\|_{L^{\f{4d}{d-2}}}^{\f4{d-2}} \|u(t_{j-1})\|_{\dot{H}^2}
      \lesssim \|v(t_{j-1})\|_{\dot{H}^1}^{\f4{d-2}}\|v(t_{j-1})\|_{\dot{H}^3}.
    \end{aligned}
  \end{align}
Substituting
  \eqref{est:h_s6} -- \eqref{est:h_s8} into \eqref{est:h_s5}, with the smallness of $\delta$, we infer that
\begin{align}\label{est:h_s9}
     \|\nabla v_t\|_{L^{\infty}_t L^2_x \bigcap L^2_t L^{\f{2d}{d-2}}_x (I_j \times \mathbb{R}^d)} & \lesssim \|v(t_{j-1})\|_{\dot{H}^3} \left(1+ \|v(t_{j-1})\|_{\dot{H}^1}^{\f4{d-2}} \right).
  \end{align}
    In particular, we have
  \begin{align*}
     \|\nabla v_t\|_{L^{\infty}_t L^2_x \bigcap L^2_t L^{\f{2d}{d-2}}_x (I_1 \times \mathbb{R}^d)} \lesssim \|v_0\|_{\dot{H}^3} \left(1+ \|v_0\|_{\dot{H}^1}^{\f4{d-2}} \right).
  \end{align*}
  Since
  \begin{align*}
    -\nabla\Delta v = -\nabla u_t - \mu \nabla f(v),
  \end{align*}
  it follows from \eqref{est:h_s8} that
  \begin{align*}
  \begin{aligned}
    \|\nabla\Delta v\|_{L^{\infty}_t L^2_x (I_1 \times \mathbb{R}^d)} & \lesssim
    \|\nabla v_t\|_{L^{\infty}_t L^2_x (I_1 \times \mathbb{R}^d)} + \left\|\nabla
    f(v)
     \right\|_{L^{\infty}_t L^2_x (I_1 \times \mathbb{R}^d)}
  \le  C\left(\|v_0\|_{\dot{H}^1} \right)\|v_0\|_{\dot{H}^3},
  \end{aligned}
  \end{align*}
  which implies
\begin{align}\label{est:h_s10}
    \|v(t_1)\|_{\dot{H}^3}
    \le     C \left(\|v_0\|_{\dot{H}^1} \right)\|v_0\|_{\dot{H}^3}.
  \end{align}
   Combining \eqref{est:h_s8}, \eqref{est:h_s9} with \eqref{est:h_s10} leads to  the estimate 
  \begin{align*}
    \|\nabla\Delta v\|_{L^{\infty}_t L^2_x (I_2 \times \mathbb{R}^d)}\le  C \left(\|v_0\|_{\dot{H}^1} \right)\|v_0\|_{\dot{H}^3}.
  \end{align*}
  Hence, by a finite induction argument for $j=1, \ldots, N$, we find
  \begin{align*}
    \|\nabla\Delta v\|_{L^{\infty}_t L^2_x (\mathbb{R}_+ \times \mathbb{R}^d)}\le  C \left(\|v_0\|_{\dot{H}^1} \right)\|v_0\|_{\dot{H}^3},
      \end{align*}
     which is \eqref{eq3.16v16} when $k = 3$.
  \end{proof}

\subsection{Proof of Theorem \ref{thm:f_schrodinger}}

Set $ \tilde{w} =  v^\theta - v $. Then $\tilde{w}$ satisfies
\begin{align}\label{eq:differen2}
  \begin{cases}
e^{-i\theta} \tilde{w}
_t -  \Delta \tilde{w}
= \mu \left( f(v^\theta) - f(v) \right)
-  \left(e^{-i\theta}-i \right) \left(\Delta v - \mu f(v)
\right), \\
\tilde{w} (0,x) = \tilde{w}_0
(x) =  v^\theta_0(x) - v_0(x).
\end{cases}
\end{align}

\begin{lemma}\label{prop:df_diff2}
  Let $d = 3, 4$. Suppose that $ \tilde{w} $ satisfies \eqref{eq:differen2}. There exists a small $\varepsilon > 0 $ such that if $\| \tilde{w}_0 \|_{H^1} \le \varepsilon$, then for each $T>0$, we have 
\begin{align}\label{est:df_w-p}
    \| \tilde{w}  \|_{S([0, T]) } \le C=C\left(T, \left\| \tilde{w}_0 \right\|_{\dot{H}^1}, \|v_0\|_{\dot{H}^3} \right).
  \end{align}
\end{lemma}
\begin{proof}
Following the argument as in the proof of Lemma \ref{prop:df_diff}, we may divide $[0, T]$ into a finite number $N$ of disjoint intervals $I_j$ 
such that the solution $v$ of \eqref{eq:s} satisfies
\begin{align}\label{est:df_w-p1}
  \|v\|_{S(I_j) }+ \|\nabla v\|_{L^{\f{2d+4}{d-2}}_t L^{\f{2d(d+2)}{d^2+4}}_x(I_j \times \mathbb{R}^d)} \le \delta_0 .
\end{align}
By the Strichartz estimate, we have
\begin{align}\label{est:df_w-p2}
\begin{aligned}
    & \left\|\nabla \tilde{w} \right\|_{L^\infty_t L^2_x(I_j \times \mathbb{R}^d)} + \left\|\nabla \tilde{w}
    \right\|_{L^{\f{2(d+2)}{d-2}}_t L^{\f{2d(d+2)}{d^2+4}}_x(I_j \times \mathbb{R}^d)} \\
    & \lesssim \left\| \tilde{w}
    (t_{j-1}) \right\|_{\dot{H}^1}+ |e^{i\theta} -i|
    \left( \left\|\nabla\Delta v \right\|_{L^1_tL^2_x(I_j \times \mathbb{R}^d)} +  \left\|\nabla
    f(v)
    \right\|_{L^2_t L_x^{\f{2d}{d+2}}(I_j \times \mathbb{R}^d)} \right)\\
    & \quad + \left\|\nabla
    \left( f \left(v^\theta \right) - f(v)
    \right) \right\|_{L^2_t L_x^{\f{2d}{d+2}}(I_j \times \mathbb{R}^d)}.
\end{aligned}
\end{align}
The terms on the right-hand side of \eqref{est:df_w-p2} can be estimated as follows. First of all, we have
\begin{align}\label{est:df_w-p3}
  \left\|\nabla\Delta v \right\|_{L^1_tL^2_x(I_j \times \mathbb{R}^d)} \le |I_j| \|v\|_{L^\infty_t\dot{H}^3_x(I_j \times \mathbb{R}^d)}.
\end{align}
Secondly, we may follow the same argument as \eqref{est:df_w3}--\eqref{est:df_w5} to obtain
\begin{align}\label{est:df_w-p4}
    \begin{aligned}
      \left\|\nabla f(v)
      \right\|_{L^2_t L_x^{\f{2d}{d+2}}(I_j \times \mathbb{R}^d)} \lesssim \delta_0^{\f{d+2}{d-2}},
    \end{aligned}
  \end{align}
  and
  \begin{align}\label{est:df_w-p5}
    \begin{aligned}
      & \left\|\nabla \left( f \left(v^\theta \right) - f(v)       \right) \right\|_{L^2_t L_x^{\f{2d}{d+2}}(I_j \times \mathbb{R}^d)} \\
      &\lesssim \left\|\nabla \tilde{w}
      \right\|_{L^{\f{2d+4}{d-2}}_t L^{\f{2d(d+2)}{d^2+4}}_x(I_j \times \mathbb{R}^d)}^{\f{d+2}{d-2}} + \delta_0^{\f4{d-2}}  \left\|\nabla \tilde{w} \right\|_{L^{\f{2d+4}{d-2}}_t L^{\f{2d(d+2)}{d^2+4}}_x(I_j \times \mathbb{R}^d)} \\
       &\quad + \delta_0  \left\|\nabla \tilde{w} \right\|_{L^{\f{2d+4}{d-2}}_t L^{\f{2d(d+2)}{d^2+4}}_x(I_j \times \mathbb{R}^d)}^{\f4{d-2}}.
    \end{aligned}
  \end{align}
  When $j=1$, $ \tilde{w}(t_0) = \tilde{w}_0$. Combining \eqref{est:df_w-p2}--\eqref{est:df_w-p5} with Proposition \ref{prop:higher_nls}, we get 
\begin{align}\label{est:df_w-p6}
    \begin{aligned}
      &\|\nabla \tilde{w}   \|_{L^\infty_t L^2_x(I_1 \times \mathbb{R}^d)} + \|\nabla \tilde{w}  \|_{L^{\f{2(d+2)}{d-2}}_t L^{\f{2d(d+2)}{d^2+4}}_x(I_1 \times \mathbb{R}^d)} \\
    & \le  C\| \tilde{w}_0 \|_{\dot{H}^1}+ C|e^{i\theta} -i|
    \left(T C \left(\|v_0\|_{\dot{H}^1} \right)\|v_0\|_{\dot{H}^3} + \delta_0^{\f{d+2}{d-2}} \right)\\
    &\quad + C\|\nabla \tilde{w}     \|_{L^{\f{2d+4}{d-2}}_t L^{\f{2d(d+2)}{d^2+4}}_x(I_1 \times \mathbb{R}^d)}^{\f{d+2}{d-2}} + C\delta_0^{\f4{d-2}}
    \|\nabla \tilde{w}     \|_{L^{\f{2d+4}{d-2}}_t L^{\f{2d(d+2)}{d^2+4}}_x(I_1 \times \mathbb{R}^d)}\\
      &\quad + C\delta_0  \|\nabla \tilde{w}  \|_{L^{\f{2d+4}{d-2}}_t L^{\f{2d(d+2)}{d^2+4}}_x(I_1 \times \mathbb{R}^d)}^{\f4{d-2}}
    \end{aligned}
  \end{align}
  with a universal constant $C$.

  We choose a small positive constant $\varepsilon$ such that $\| \tilde{w}_0 \|_{\dot{H}^1} \le \f\varepsilon2$ and $|e^{i\theta}-i| T \le \f\varepsilon4$. By a standard continuity method, \eqref{est:df_w-p6} is reduced to
\begin{align}\label{est:df_w-p7}
    \|\nabla \tilde{w} \|_{L^\infty_t L^2_x(I_1 \times \mathbb{R}^d)} + \|\nabla \tilde{w} \|_{L^{\f{2(d+2)}{d-2}}_t L^{\f{2d(d+2)}{d^2+4}}_x(I_1 \times \mathbb{R}^d)} \le 2C\varepsilon,
  \end{align}
 which also gives the bound $\|\nabla \tilde{w}
    (t_1)\|_{L^2} \le 2C\varepsilon$. Hence by a finite induction argument corresponding to the finite partition of $[0, T]$, we may deduce
  \begin{align*}
    \|\nabla \tilde{w} \|_{L^\infty_t L^2_x(I_N \times \mathbb{R}^d)} + \|\nabla \tilde{w}  \|_{L^{\f{2(d+2)}{d-2}}_t L^{\f{2d(d+2)}{d^2+4}}_x(I_N \times \mathbb{R}^d)} \le (4C)^N\varepsilon.
  \end{align*}
  If we restrict $\varepsilon$ to be small such that $(4C)^N\varepsilon \le \delta$, then we could derive
  \begin{align*}
    \|\nabla \tilde{w} \|_{L^\infty_t L^2_x([0,T] \times \mathbb{R}^d)} + \|\nabla \tilde{w}  \|_{L^{\f{2(d+2)}{d-2}}_t L^{\f{2d(d+2)}{d^2+4}}_x([0,T] \times \mathbb{R}^d)} \le C    \left(T,\| \tilde{w}_0\|_{\dot{H}^1}, \|v_0\|_{\dot{H}^3}\right).
  \end{align*}
The desired estimate \eqref{est:df_w-p} then follows from Sobolev's inequality.
\end{proof}

We can now give the proof of Theorem \ref{thm:f_schrodinger}. Since a large part of the proof  is very similar to that used in Theorem \ref{thm:df_heat}, we shall only give a sketch here and point out the differences.

\begin{proof}[Sketch for the proof of Theorem \ref{thm:f_schrodinger}.]
We may follow the proof of Theorem \ref{thm:df_heat}, except that we use
\begin{itemize}
\item[i)] $|e^{i\theta} - i|$ to replace $|e^{i \theta} - 1| $ in \eqref{est:main1_1}, \eqref{est:main1_3}, \eqref{est:main1_4}, \eqref{est:main1_5}, \eqref{eq3.33v21};

\item[ii)] $|I_j|\|\Delta u\|_{L^\infty_tL^2_x(I_j \times \mathbb{R}^d)}$ to replace $|I_j|^{1/2} \|\Delta u\|_{L^2_tL^2_x(I_j \times \mathbb{R}^d)}$ in \eqref{est:main1_3}, \eqref{est:main1_4}, \eqref{est:main1_5};

\item[iii)] $T \|\Delta u\|_{L^\infty_tL^2_x([0, T] \times \mathbb{R}^d)}$ to replace $T^{1/2} \|\Delta u\|_{L^2_tL^2_x([0, T] \times \mathbb{R}^d)} $ in \eqref{eq3.33v21};

\item[iv)] Lemma \ref{prop:df_diff2} for the NLS instead of Lemma \ref{prop:df_diff} for the NLH.
\end{itemize}
\end{proof}
\medskip 

\noindent \textbf{Acknowledgments. }
The authors are grateful to Prof.~Lifeng Zhao for very helpful discussions on related topics.

\end{document}